\documentclass[11pt]{amsart}

\usepackage[latin1]{inputenc}
\usepackage{graphicx}

\newtheorem{teo}{Theorem}[section]

\newtheorem{lema}[teo]{Lemma}

\theoremstyle{remark}
\newtheorem{rem}{Remark}[section]

\def\ve{\varepsilon}
\def\R{{\mathbb{R}}}
\def\N{{\mathbb{N}}}
\def\s{\sigma}

\def\E{\mathbb{E}}
\def\P{\mathbb{P}}

\def\F{\mathcal F}

\def\one{\mathbf 1}

\begin{document}
\title[Reaction diffusion equations with additive noise]
{Time--space white noise eliminates global solutions in reaction
diffusion equations}

\author{Julian Fern\'{a}ndez Bonder \and Pablo Groisman }
\thanks{Supported by Universidad de Buenos Aires under grants X447 and X078,
by ANPCyT PICT 2006--290 and 2006--1309 and CONICET PIP 5478/1438. Both authors
are members of CONICET.
}


\address{J. Fern\'{a}ndez Bonder and P. Groisman \hfill \break \indent
Departamento  de Matem\'atica, FCEyN, Universidad de Buenos Aires,
\hfill\break\indent Pabell\'{o}n I, Ciudad Universitaria (1428),
Buenos Aires, Argentina.}

\email{{\tt jfbonder@dm.uba.ar, pgroisma@dm.uba.ar}\hfill \break
\indent {\it URL:} {\tt http://mate.dm.uba.ar/$\sim$jfbonder,
http://mate.dm.uba.ar/$\sim$pgroisma}}

\begin{abstract}
We prove that perturbing the reaction--diffusion equation
$u_t=u_{xx} + (u_+)^p$ ($p>1$), with time--space white noise
produces that solutions explodes with probability one for every
initial datum, opposite to the deterministic model where a positive
stationary solution exists.
\end{abstract}

\maketitle

\section{Introduction}
\setcounter{equation}{0}

In this paper we study the following parabolic SPDE with additive
noise
\begin{equation}\label{SPDE}
u_t = u_{xx} + f(u) + \s \dot{W}(x,t),
\end{equation}
in an interval $(0,1)$, complemented with homogeneous Dirichlet
boundary conditions. Here $W$ is a $2-$dimensional Brownian sheet,
$\s$ is a positive parameter and $f$ is a locally Lipschitz real
function.

We restrict ourselves to one space dimension since for higher
dimensions the solution to \eqref{SPDE} (if it exists) it is not
expected to be a function valued process and have to be understood
in a distributional sense. But in this case there is no natural way
to define $f(u)$, see \cite{Pardoux} for more on this.

Semilinear parabolic equations like \eqref{SPDE} arises in the phenomenological
approach to such different phenomena as the diffusion of a fluid in a porous
medium, transport in a semiconductor, chemical reactions with possibility of
spatial diffusion, population dynamics, chemotaxis in biological systems, etc.
In all these cases, due to the phenomenological approximate character of the
equations, it is of interest to test how the description changes under the
effect of stochastic perturbation.

Equation \eqref{SPDE} with $f$ globally Lipschitz has been widely studied (see
\cite{Pardoux, Walsh}), in this case global solutions exist with probability
one. However, when $f$ is just locally Lipschitz, typically $f(s) \sim s^p$
with $p>1$ or $f(s) \sim {\rm e}^s$, there are practically no results on this
problem. Using standard approximation arguments one can easily prove the
existence of local in time solutions but it does not follow from that proof the
behavior of the maximal time of existence.

On the other hand, the deterministic case (i.e. $\s=0$) is very well
understood. One problem that has drawn the attention to the PDE community is
the appearance of singularities in finite time, no matter how smooth the
initial data is. This phenomena is known as {\em blow-up}. What happens is that
solutions go to infinity in finite time, that is, there exists a time
$T<\infty$ such that
$$
\lim_{t\nearrow T} \|u(\cdot,t)\|_{\infty} = \infty.
$$
A well known condition on the nonlinear term $f$ that assures this phenomena is
when $f$ is a nonnegative convex function with
$$
\int^\infty \frac{1}{f} < \infty.
$$
For a general reference of these facts and much more on blow-up problems, see
the book \cite{SGKM} and the surveys \cite{BB, JLV}.

For a large class of nonlinearities $f$, such as the ones mentioned above,
problem \eqref{SPDE} with $\s=0$ admits a stationary positive solution $v$ and
hence, since the comparison principle holds for this equation, for every
initial datum $u_0 \le v$ the solution to \eqref{SPDE} is global in time.

It is well known (see \cite{JLV,SGKM}) that the appearance of blow-up persists
under (small) regular perturbations. On the other hand, regular perturbations
of \eqref{SPDE} with $\s=0$ admit global in time solutions. Summarizing, the
existence of global in time/blowing up solutions for this problem with $\s=0$
is stable under small regular perturbations. Hence it is of interest to test
how this phenomena is affected by stochastic perturbations.

Surprisingly, the situation changes for $\s >0$. We prove that, in this case,
there is no global in time solution. In fact, for every initial nonnegative
datum $u_0$, the solution to \eqref{SPDE} blows up with probability one.

Stochastic partial differential equations with blow-up has been considered by
C. Mueller in \cite{M1, M2} and C. Mueller and R. Sowers in \cite{MS}. In those
papers, a linear drift with a nonlinear multiplicative noise is considered and
the explosion is due to this latter term.

A similar result, but in some sense in the opposite direction, was proved by
Mao, Marion and Renshaw in \cite{MMR}. There, the authors prove for a system of
ODEs that arise in population dynamics and that have blow-up solutions, that
perturbing some coefficients of the system with a small Brownian noise, global
solutions a.s. are obtained for every initial data.

In our problem, a common way to interpret the asymptotic behavior of $u$ is the
following: consider first the deterministic case $\s=0$. In this case there is
some kind of competition between the diffusion, which diffuses the zero
boundary condition to the interior of the domain and the nonlinear source
$f(u)$ that induces $u$ to grow very fast.

Again in the deterministic case, it was proved in \cite{CE} that for small
initial datum $u_0$, $u \to 0$ as $t \to +\infty$, while for $u_0$ large, there
exists a finite time $T$, such that $\|u(\cdot, t)\|_\infty \nearrow +\infty$
as $t \nearrow T$. More precisely, it is proved that for every data $u_0$,
there exists a critical parameter $\lambda^*$ such that if we solve the PDE
with initial data $\lambda u_0$, for $\lambda<\lambda^*$ the solution converges
to $0$ uniformly, for $\lambda>\lambda^*$ the solution blows-up in finite time
and for $\lambda=\lambda^*$ the solution converges uniformly to the unique
positive steady state.

For small noise $\s \ll 1$ one could expect a similar behavior. Of course we
can not expect convergence to the zero solution as $t \to \infty$ since in this
case $v\equiv 0$ is not invariant for \eqref{SPDE}, but it is reasonable to
suspect the existence of an invariant measure close to the zero solution of the
deterministic PDE and convergence to this invariant measure for small initial
datum as $t \to \infty$.

However, that is not the case. We prove in Section \ref{explosions} that for
every initial datum $u_0$ solutions to \eqref{SPDE} blow-up in finite time with
probability one.

Numerical simulations, as well as heuristical arguments, suggest that, for
small initial data $u_0$, metastability could be taking place in this case.
Metastability appears here since, while the noise remains relatively small, the
solution stays in the domain of attraction of the zero solution of the
deterministic problem. But, as soon as the noise becomes large, the solution
escapes this domain of attraction and hence the reaction term begins to
dominate and pushes forward the solution until ultimately explosion cannot be
prevented by the action of the noise.

\subsection*{Organization of the paper}
The paper is organized as follows. In Section \ref{rigourous} we give the
rigorous meaning of \eqref{SPDE} and give the references where the foundations
for the study of this kind of equation were laid. Section \ref{explosions}
deals with the proof of the main result of this paper: the explosion of the
solutions of \eqref{SPDE}. In Section \ref{approximations} we propose a
semidiscrete scheme in order to approximate the solutions to \eqref{SPDE}. We prove that the numerical approximations also explode with probability one and that they converge a.s., in time intervals where the
continuous solution remains bounded. Finally, in Section \ref{experiments} we
show some numerical simulations for this equation.

\section{Formulation of the problem}
\label{rigourous}
\setcounter{equation}{0}

We begin this section discussing the rigorous meaning of \eqref{SPDE}, the
references for this being \cite{BP,GP, Pardoux, Walsh}. There are two
alternatives: the {\em integral} and the {\em weak} formulation as described in
\cite{BP, Pardoux, Walsh}. The last being more suitable for our purposes. Both
formulations are equivalent as is shown in \cite{Walsh}.

Let $(\Omega, \F, (\F_t)_{t\ge 0}, \P)$ be a probability space equipped with a
filtration $(\F_t)_{t\ge 0}$ which is supposed to be right continuous and such
that $\F_0$ contains all the $\P-$null sets of $\F$. We are given a space-time
white noise on $\R_+ \times [0,1]$ defined on $(\Omega, \F, (\F_t)_{t\ge 0},
\P)$ and $u_0 \in C_0([0,1])$.

Assume for a moment that $f$ is globally Lipschitz, multiply \eqref{SPDE} by a
test function $\varphi \in C^2((0,1))\cap C_0([0,1])$ and integrate to obtain
\begin{equation}
\label{debil}
\begin{aligned}
\int_0^1 &u(x,t) \varphi(x) \, dx  - \int_0^1 u_0(x)\varphi(x) \, dx =\\
&\int_0^t\int_0^1 u(s,x) \varphi_{xx}(x)\, dx \, ds
 + \int_0^t\int_0^1 f(u(s,x)) \varphi(x)\, dx \, ds\\
 & + \s \int_0^t\int_0^1 \varphi(x)\, dW(x,s).
\end{aligned}
\end{equation}

Alternatively, the integral formulation of the problem is constructed by means
of the function $G$, the fundamental solution of the heat equation for the
domain $(0,1)$.
\begin{align*}
 u(x,t) - &\int_0^1 G_t(x,y) u_0(y) \, dy =\\
&\int_0^t \int_0^1 G_{t-s}(x,y) f(u(y,s)) \, dy ds + \s \int_0^t
\int_0^1 G_{t-s}(x,y) dW(y,s).
\end{align*}

As a solution to \eqref{SPDE} we understand an $\F_t-$adapted
process with values in $C_0([0,1])$ that verifies \eqref{debil} for
every $\varphi \in C^\infty((0,1)) \cap C_0([0,1])$.

In \cite{BP, Walsh} it is proved that there exists a unique solution
to this problem and that the integral and weak formulations are
equivalent.

For $f$ locally Lipschitz globally defined solutions do not exist in
general. Nevertheless, existence of local in time solutions is
proved by standard arguments: consider for each $n \in \N$ the
globally Lipschitz function $f_n(x)=f(-n)\one_{(-\infty,-n]} +
f(x)\one_{(-n,n)} + f(n)\one_{[n,+\infty)}$ and $u^n$, the unique
solution of \eqref{SPDE} with $f$ replaced by $f_n$. Let $T_n$ be
the first time at which $\|u^n(\cdot,t)\|_{\infty}$ reaches the
value $n$. Then $(T_n)_n$ is an increasing sequence of stopping
times and we define the maximal existence time of \eqref{SPDE} as
$T:=\lim T_n$. It is easy to see that $u^{n+1} \one_{\{t<T_n\}} =
u^n \one_{\{t<T_n\}}$ a.s. and hence there exist the limit
$u(x,t)=\lim u^n(x,t)$ for $t<T$ which verifies
\begin{equation}
\begin{aligned}
\int_0^1 &u(x,t\wedge T) \varphi(x) \, dx  - \int_0^1 u_0(x)\varphi(x) \, dx =\\
&\int_0^{t\wedge T}\int_0^1 u(s,x) \varphi_{xx}(x)\, dx \, ds
+ \int_0^{t\wedge T}\int_0^1 f(u(s,x)) \varphi(x)\, dx \, ds\\
&  + \s \int_0^{t\wedge T}\int_0^1 \varphi(x)\, dW(x,s).
\end{aligned}
\end{equation}

So we say that $u$ solves \eqref{SPDE} up to the explosion time $T$.
We also say that $u$ blows up in finite time if $\P(T<\infty)
> 0$. Observe that if $T(\omega)<\infty$ then
$$
\lim_{t\nearrow T(\omega)} \|u(\cdot, t, \omega)\|_\infty = \infty.
$$

\section{Explosions}
\label{explosions}
\setcounter{equation}{0}

 In this section, we show that equation
\eqref{SPDE} blows-up in finite time with probability one for every
initial datum $u_0 \in C_0([0,1])$. Hereafter we assume that $f$ is a nonnegative convex function, hence locally Lipschitz. Moreover we assume that $\int^\infty 1/f < \infty$.

In order to prove the blow-up of $u$, we define the function
$$
\Phi(t):=\int_0^1 \phi(x) u(x,t)\, dx.
$$
Here $\phi(x)>0$ is the normalized first eigenfunction of the
Dirichlet Laplacian in $(0,1)$. That is,
$\phi(x)=\tfrac{\pi}{2}\sin(\pi x)$ and hence we can use it as a
test function in \eqref{debil} to obtain

\begin{align*}
\Phi(t) -  \Phi(0) =& -\lambda_1 \int_0^t\Phi(s)\, ds + \int_0^t\int_0^1 \phi(x)
f(u(x,s))\, dx ds \\
&+ \s \int_0^t\int_0^1 \phi(x)\, dW(x,s).
\end{align*}

We denote by $z_0:=\Phi(0)=\int_0^1 \phi(x) u_0(x)\, dx$.

Now, as $f$ is convex, by Jensen's inequality, we get
$$
\int_0^1 \phi(x) f(u(x,s))\, dx \ge f\Big(\int_0^1 \phi(x) u(x,s)\,
dx\Big) = f(\Phi(s)).
$$
Moreover, since $\phi$ is a positive function with $L^1-$norm equal
to 1, it is easy to see that
$$
B(t) := \frac{\sqrt{8}}{\pi}\int_0^t\int_0^1 \phi(x)\, dW(x,s),
$$
is a standard Brownian motion.

Combining all these facts, we obtain that $\Phi$ verifies the (one
dimensional) stochastic differential inequality
$$
d\Phi(t) \ge \big(-\lambda_1  \Phi(t) + f(\Phi(t))\big) \, dt +
\frac{\pi}{\sqrt{8}}\s dB(t).
$$

Define $z(t)$ to be the one-dimensional process that verifies
$$
dz = (-\lambda_1 z + f(z))\, dt + \s dB,
$$
with initial condition $z(0)=z_0$. Then, $e(t) = \Phi(t) - z(t)$
verifies
$$
de \ge \Big(-\lambda_1 e + \frac{f(\Phi)-f(z)}{\Phi-z} e\Big)\, dt.
$$
Observe that $e$ verifies a deterministic differential inequality.
Hence, as $e(0)=0$ it is easy to check that $e(t) \ge 0$ as long as it
is defined.

Therefore, $\Phi(t)\ge z(t)$ as long as $\Phi$ is defined.

The following lemma proves that $z$ explodes with probability one.

\begin{lema}
\label{feller.test}
Let $z$ be the solution of
\begin{equation}
\label{1.d} dz = (-\lambda_1 z + f(z))\, dt + \s dB, \qquad z(0)=0.
\end{equation}
Then $z$ explodes in finite time with probability one.
\end{lema}

\begin{proof}
The proof is just an application of the {\em Feller Test for
explosions} (\cite{KS}, Chapter 5). Using the same notation as in
\cite{KS} we obtain the scale function for \eqref{1.d} to be
$$
p(x) = \int_0^x \exp\left(-\frac{2}{\s^2}\int_0^s b(\xi) \,
d\xi\right) \, ds
$$
Here $b(\xi)= -\lambda_1 \xi + f(\xi)$.

It is easy to see that, as $\int^\infty 1/f <\infty$,
$$
p(-\infty)= - \infty, \qquad p(+\infty) < +\infty,
$$
and hence the Feller Test implies that, if $S$ is the explosion time
of $z$, we get
$$
\P \left (\lim_{t \nearrow S} z(t) = + \infty \right ) = 1
$$
To prove that $\P(S < +\infty ) = 1$ we have to consider the
function
$$
v(x) =  2 \int_0^x \frac{p(x) - p(y)}{\s^2 p(y) } \, dy.
$$
The behavior of $v$ at $+\infty$ is given by $1/f$ and hence
$v(+\infty)<+\infty$, which implies that
$$
\P(S < \infty ) = 1.
$$
This completes the proof.
\end{proof}

These facts all together, imply that there exists a (random) time
$T=T(\omega)<\infty$ a.s. such that
$$
\lim_{t\nearrow T} \|u(\cdot,t)\|_{\infty} = \infty \quad
\text{a.s.}
$$

So we have proved the following Theorem.

\begin{teo}\label{teo.explota.continuo}
Let $f$ be a nonnegative, convex function such
that
$$
\int^\infty \frac{1}{f}\, <\infty.
$$
Then, for every nonnegative initial datum $u_0 \ge 0$ the solution $u$ to
\eqref{SPDE} blows-up in finite (random) time $T$ with
$$
\P^{u_0}(T<\infty) = 1.
$$
\end{teo}

\section{Numerical approximations}
\label{approximations} \setcounter{equation}{0}

In this section we introduce a numerical scheme in order to compute solutions
to problem \eqref{SPDE}. We discretize the space variable with second order
finite differences in a uniform mesh of size $h=1/n$. That is, for $x:=i/n,
i=1, 2, \dots, n-1$ the process $u^n(t,i/n)=u_i(t)$ is defined as the solution
of the system of stochastic differential equations

\begin{equation}
\label{semidiscrete}
du_i = \frac{1}{h^2} (u_{i+1} -2u_i + u_{i-1}) dt + f(u_i) \, dt +
\frac{\s}{\sqrt{h}} \, dw_i, \quad 2 \le i \le n-1,
\end{equation}
accompanied with the boundary conditions $u_1(t)=u_n(t)=0$,
$u_i(0)=u_0(ih)$, $1 \le i \le n$. The Brownian motions $w_i$ are
obtained by space integration of the Brownian sheet in the interval
$[ih, (i+1)h)$.

Equivalently, this can be written as
$$
dU = (-AU + f(U))\, dt + \frac{\s}{\sqrt{h}} \, dW,
U(0)=U^0.
$$
Where $U(t)=(u_1(t), \dots,u_n(t))$, $-A$ is the discrete laplacian, $f(U)$ in
understood componentwise (i.e. $f(U)_i = f(u_i))$, $dW=(dw_1, \dots, dw_n)$ and
$(U^0)_i=u_0(ih)$.

With the same techniques of Theorem \ref{teo.explota.continuo} it can be proved
that solutions to this system of SDEs explodes in finite time with probability
one.

We extend $u^n(t,\cdot)$ to the whole interval $[0,1]$ by linear interpolation
in the space variable for each $t$.

Concerning the explosions of this system of SDEs we have the following

\begin{teo}\label{teo.explota.semidiscreto}
Let $f$ be a nonnegative, convex function such
that
$$
\int^\infty \frac{1}{f}\, <\infty.
$$
Then, for every nonnegative initial datum $U^0 \ge 0$ the solution $U$ to
\eqref{semidiscrete} blows-up in finite (random) time $T^n$ with
$$
\P^{U^0}(T^n<\infty) = 1.
$$
\end{teo}

\begin{proof}
The proof uses the same technique of that of Theorem \ref{teo.explota.continuo}.
Since $A$ is a symmetric positive definite matrix, we have a sequence of
positive eigenvalues of $A$, $0 < \lambda^n_1 \le \dots \le \lambda_n^n$. Let
$\phi^n$ the eigenvector associated to $\lambda^n_1$. It is easy to see that
one can tale $\phi^n$ such that $\phi^n_j\ge 0$ for every $j$, and we assume
that it is normalized such that $\sum_{i=1}^n h \phi_i^n=1$. Now, consider the function
$$
\Phi^n(t) = \sum_{i=1}^n  h\phi^n_i u_i(t).
$$
Proceeding as in the proof of Theorem \ref{teo.explota.continuo} we get that
$\Phi^n$ verifies

$$
d\Phi^n(t)  \ge (-\lambda^n_1\Phi^n(t) + f(\Phi^n(t))) \, dt + \s_n dB(t),
$$
where $B$ is a standard Brownian motion and $\s_n \to \s\pi/\sqrt{8}$. The rest of the proof follows by Lemma
\ref{feller.test} as in Theorem \ref{teo.explota.continuo}.
\end{proof}

Now we turn to the problem of convergence of the approximations. In \cite{Gy}
convergence of this numerical scheme for globally Lipschitz reactions is proved

\begin{teo}[Gy\"ongy, \cite{Gy} Theorem 3.1]
\label{teo.gyongy}
Assume $f$ is globally Lipschitz and $u_0 \in C^3([0,1])$. Then
\begin{enumerate}
\item For every $p\ge 1$ and for every $T>0$ there exists a constant $K=K(p,T)$
such that
$$
\sup_{0 \le t \le T}\sup_{x \in [0,1]} \E (|u^n(t,x) - u(t,x)|^{2p}) \le
\frac{K}{n^p}.
$$

\item $u^n(t,x)$ converges to $u(t,x)$ uniformly in $[0,T]\times[0,1]$ almost
surely as $n\to \infty$.
\end{enumerate}
\end{teo}

Based on this theorem we can prove that even when $f$ is just locally
Lipschitz, convergence holds but just in (stochastic) time intervals where the
solution remains bounded. Observe that a better convergence result is not
expected. Since the explosion times of $u$ and $u^n$ in general are different,
then $\|u^n(t,\cdot) - u(t, \cdot)\|_\infty$ is unbounded in intervals of the
form $[0, \tau]$ with $\tau$ close to the minimum of the explosion times. To
state the convergence result we define the following stopping times. Let $M>0$
and consider $R_M:=\inf\{t>0, \|u(t,\cdot)\|_{L^\infty([0,1])} \ge M \}$ and
$R^n_M:=\inf\{t>0, \|u^n(t,\cdot)\|_{L^\infty([0,1])} \ge M \}$

\begin{teo}
Assume $f$ is a nonnegative convex function with $\int \tfrac 1 f < \infty$.
Let $u$ be the solution to \eqref{SPDE} and $u^n$ its numerical approximation
given by \eqref{semidiscrete}. Then
\begin{enumerate}
\item For every $p\ge 1$ and for every $T>0$ there exists a constant $K=K(p,T)$
such that
$$
\sup_{0 \le t \le T}\sup_{x \in [0,1]} \E (|u^n(t,x) - u(t,x)|^{2p}\one_{\{t\le R_M\wedge R^n_M\}}) \le \frac{K}{n^p}.
$$

\item For every $M\ge 0$ $\|u^n - u\|_{L^\infty([0,T \wedge R_M]\times[0,1])}$
converges to zero almost surely as $n\to \infty$.
\end{enumerate}
\end{teo}

\begin{rem}
Observe that statement (2) does not make assumptions on the numerical
approximations $u^n$.
\end{rem}

\begin{proof}
First, we truncate the $f$ to get a globally Lipschitz function, bounded and
that coincides with the original $f$ for values of $s$ with $|s|\le M$. i.e. we
consider
\begin{align*}
f_M(s) &= \left\{ \begin{array}{ll}
                    f(s) & \mbox{if } |s|\le M\\
                    f(M) & \mbox{if } s\ge M\\
                    f(-M)& \mbox{if } s\le -M,
                    \end{array}\right.\\
\end{align*}
Let $w$ and $w^n$ be the solutions of \eqref{SPDE} and \eqref{semidiscrete}
with $f$ replaced by $f_M$ respectively.

{}From Theorem \ref{teo.gyongy},
$$
\sup_{0 \le t \le T}\sup_{x \in [0,1]} \E (|w^n(t,x) - w(t,x)|^{2p}) \le
\frac{K}{n^p},
$$
From the uniqueness of solutions of\eqref{SPDE} and \eqref{semidiscrete} up to
the stopping time $R_M \wedge R_M^n$, we have that almos surely, if $t\le R_M \wedge R_M^n$
then $u(t,x)=w(t,x)$ and $u^n(t,x)=w^n(t,x)$, hence
\begin{align*}
\sup_{0 \le t \le T}\sup_{x \in [0,1]} & \E (|u^n(t,x) - u(t,x)|^{2p}\one_{\{t\le R_M\wedge R^n_M\}})  = \\
& \sup_{0 \le t \le T}\sup_{x \in [0,1]} \E (|w^n(t,x) - w(t,x)|^{2p}\one_{\{t\le R_M\wedge R^n_M\}}) \le\\
& \sup_{0 \le t \le T}\sup_{x \in [0,1]} \E (|w^n(t,x) - w(t,x)|^{2p})  \le \frac{K}{n^p}.
\end{align*}

This proves (1). To prove (2) observe that since $w^n \to w$ almost surely and
uniformly in $[0,T]\times[0,1]$ we have that for every $\ve>0$ and $0\le t\le
R_M$, $\|w^n(t, \cdot)\|_\infty \le M+\ve$ if $n$ is large enough. That means
that $\liminf R_M^n\ge R_M$ and hence $R_M\wedge R_M^n \to R_M$. That is the
reason we can get rid of $R_M^n$. So we have
\begin{align*}
0=& \lim_{n\to \infty}\|w^n - w\|_{L^\infty([0,T]\times[0,1])}\\
\ge &\lim_{n\to \infty}\|(w^n - w)\one_{\{t\le R_{M-1}\wedge R^n_{M}\}}\|_{L^\infty([0,T]\times[0,1])}\\
= & \lim_{n\to \infty}\|(u^n - u)\one_{\{t\le R_{M-1}\wedge R^n_{M}\}}\|_{L^\infty([0,T]\times[0,1])}\\
\ge & \lim_{n\to \infty}\|(u^n - u)\one_{\{t\le R_{M-1}\}}\|_{L^\infty([0,T]\times[0,1])}\\
= & \lim_{n\to \infty}\|u^n - u\|_{L^\infty([0,T\wedge R_{M-1}]\times[0,1])}.
\end{align*}
Since $M$ is an arbitrary constant, this proves (2).
\end{proof}

\begin{rem}
In order to compute an approximate solution this discretization in not enough,
now we need to discretize the time variable but this is much simpler since now
we are dealing with a SDE instead of a SPDE. The time discretization of
\eqref{semidiscrete} can be handled as in \cite{DFBGRS}.
\end{rem}

\section{Numerical experiments}
\label{experiments}
\setcounter{equation}{0}

In this section we show some numerical simulations of \eqref{SPDE}. We perform
all the simulations with the reaction $f(u)=(u_+)^2$, $\s=6.36$ and initial
datum $u_0\equiv0$.

\begin{figure}[ht]
$$
\begin{array}{ccc}
\includegraphics[width=4cm]{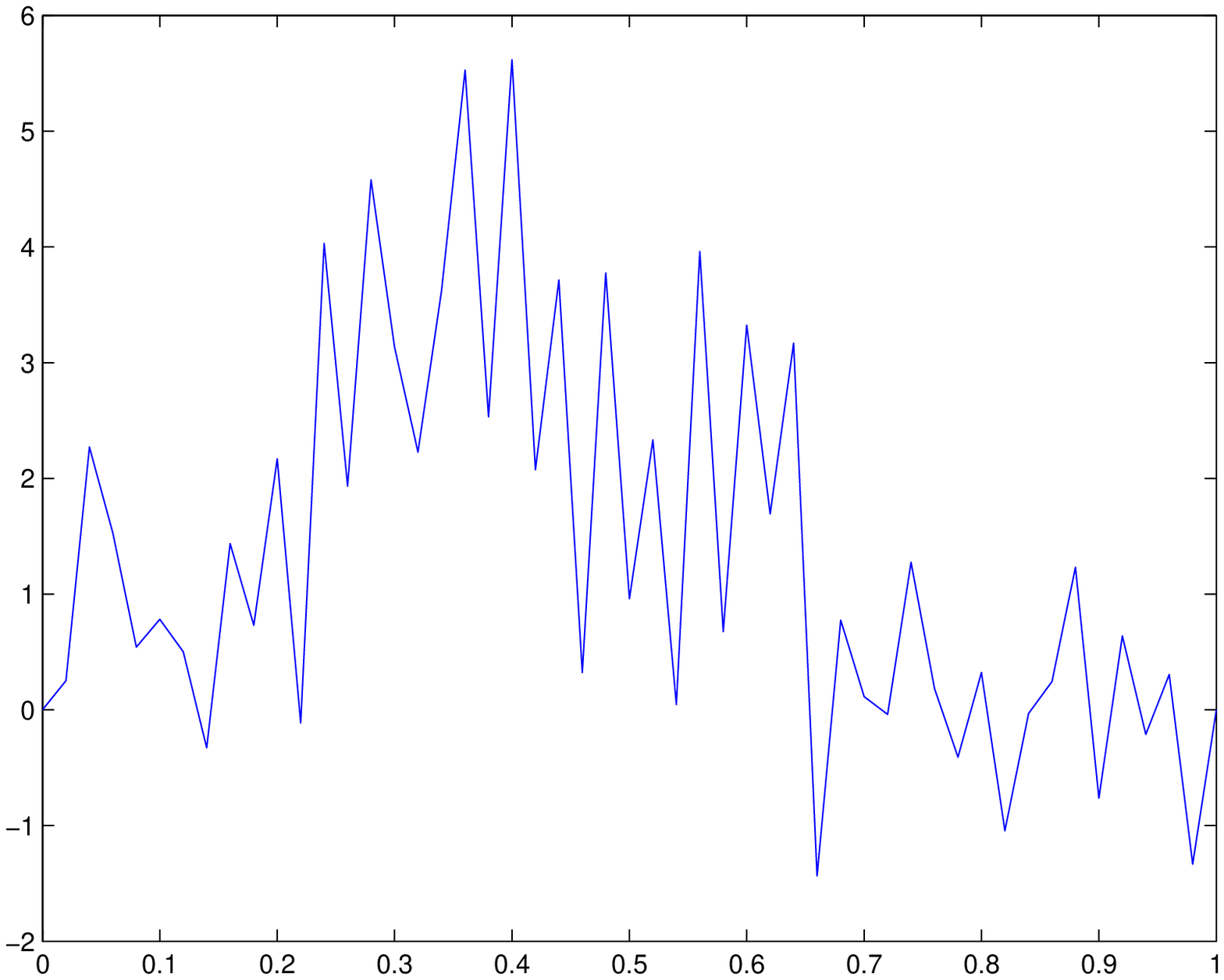} & \hspace{-6pt} \includegraphics[width=4cm]{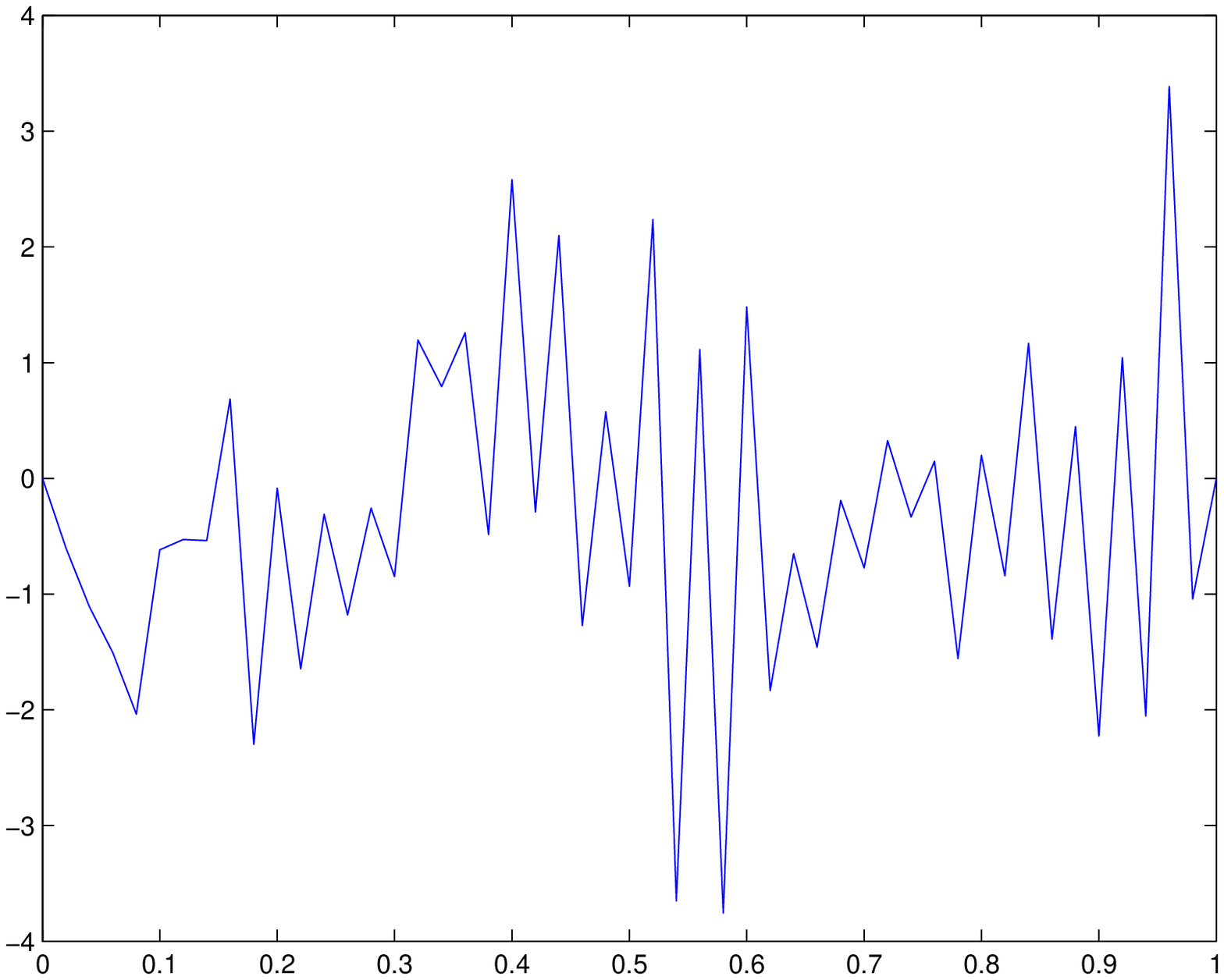} &
\hspace{-6pt}\includegraphics[width=4cm]{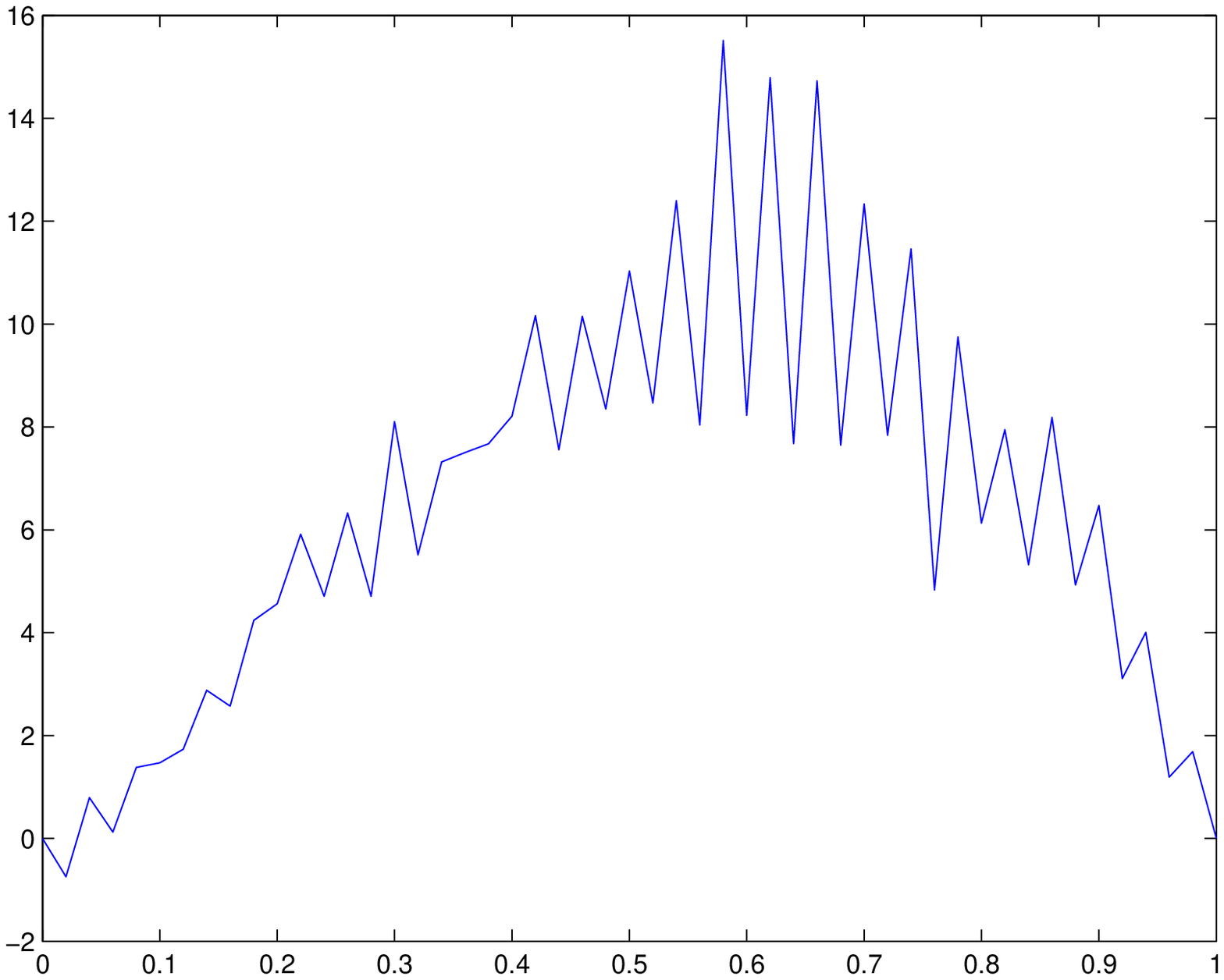}\\[-6pt]
\scriptstyle t=1 & \scriptstyle t=50 & \scriptstyle t=72.0202\\
\includegraphics[width=4cm]{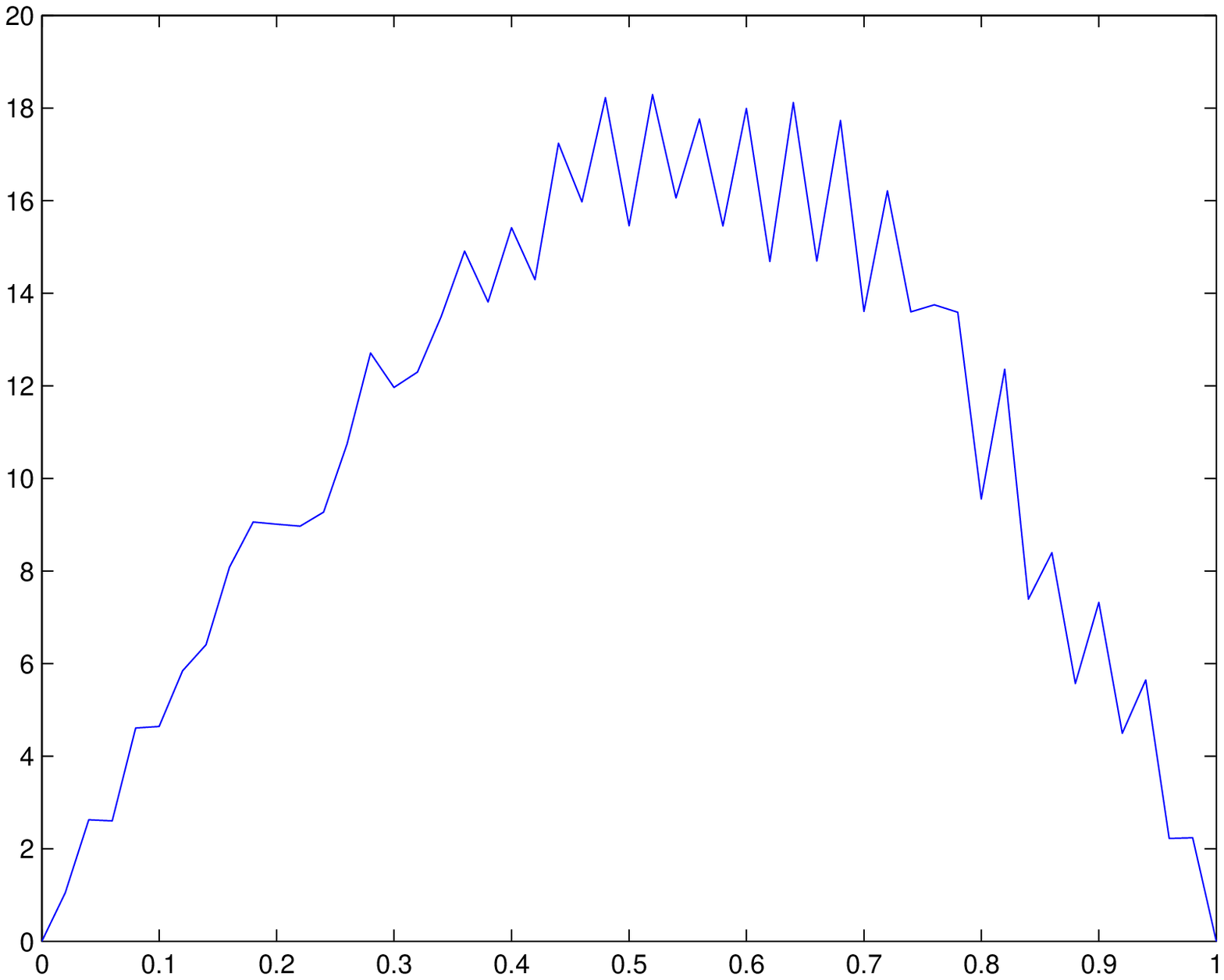} & \hspace{-6pt} \includegraphics[width=4cm]{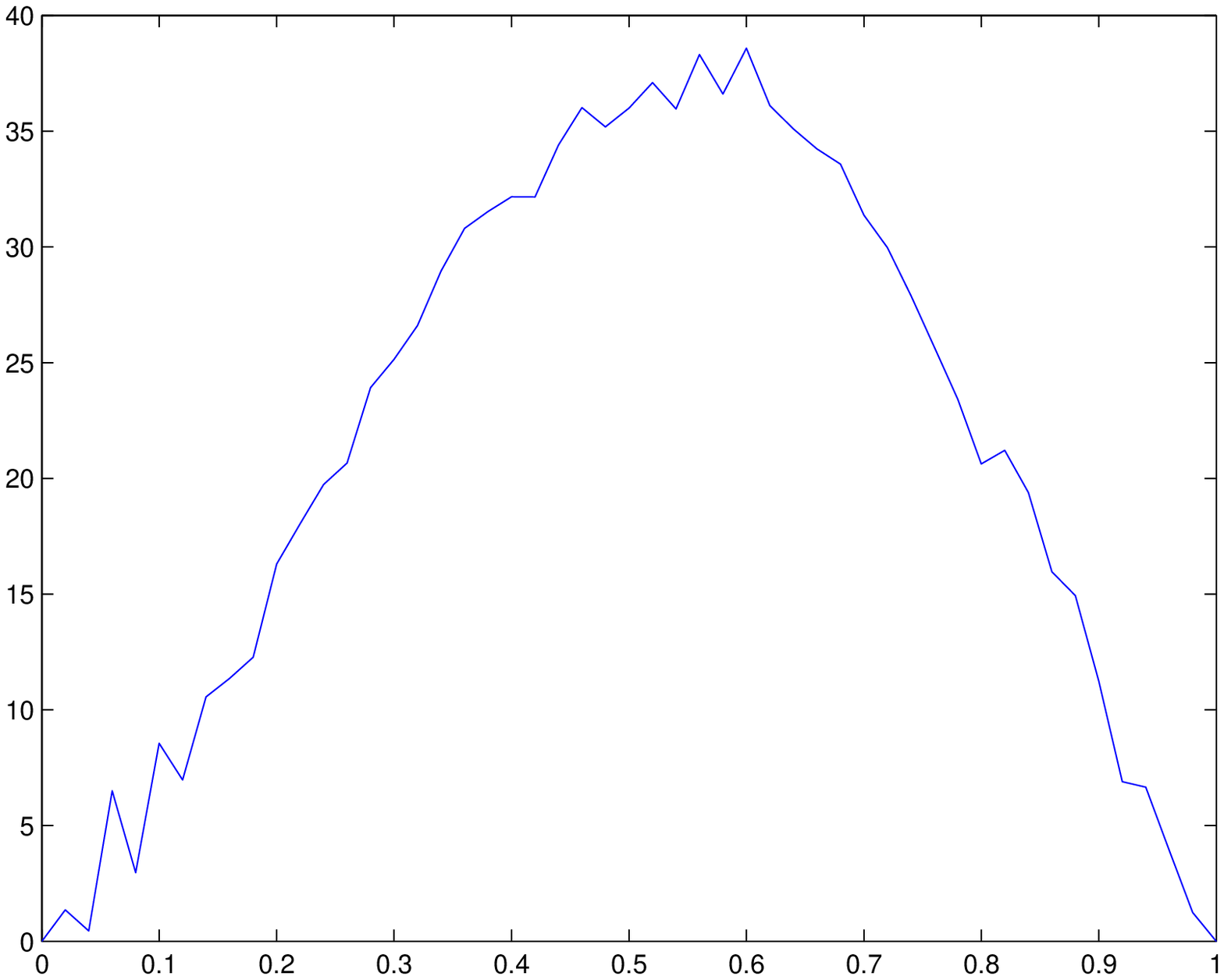} &
\hspace{-6pt}\includegraphics[width=4cm]{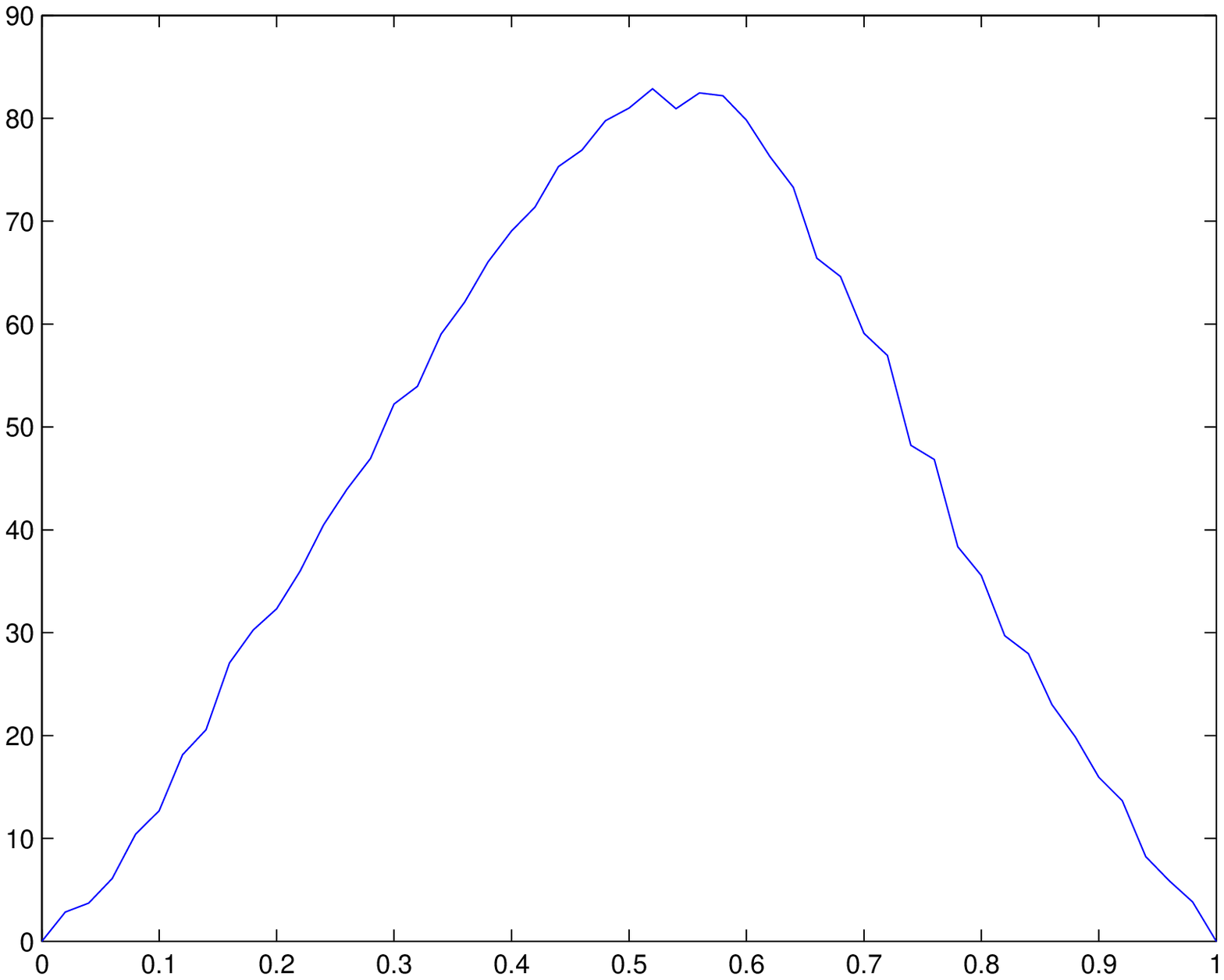}\\[-6pt]
\scriptstyle t=72.4202 & \scriptstyle t=72.4802& \scriptstyle t=72.5002\\
\includegraphics[width=4cm]{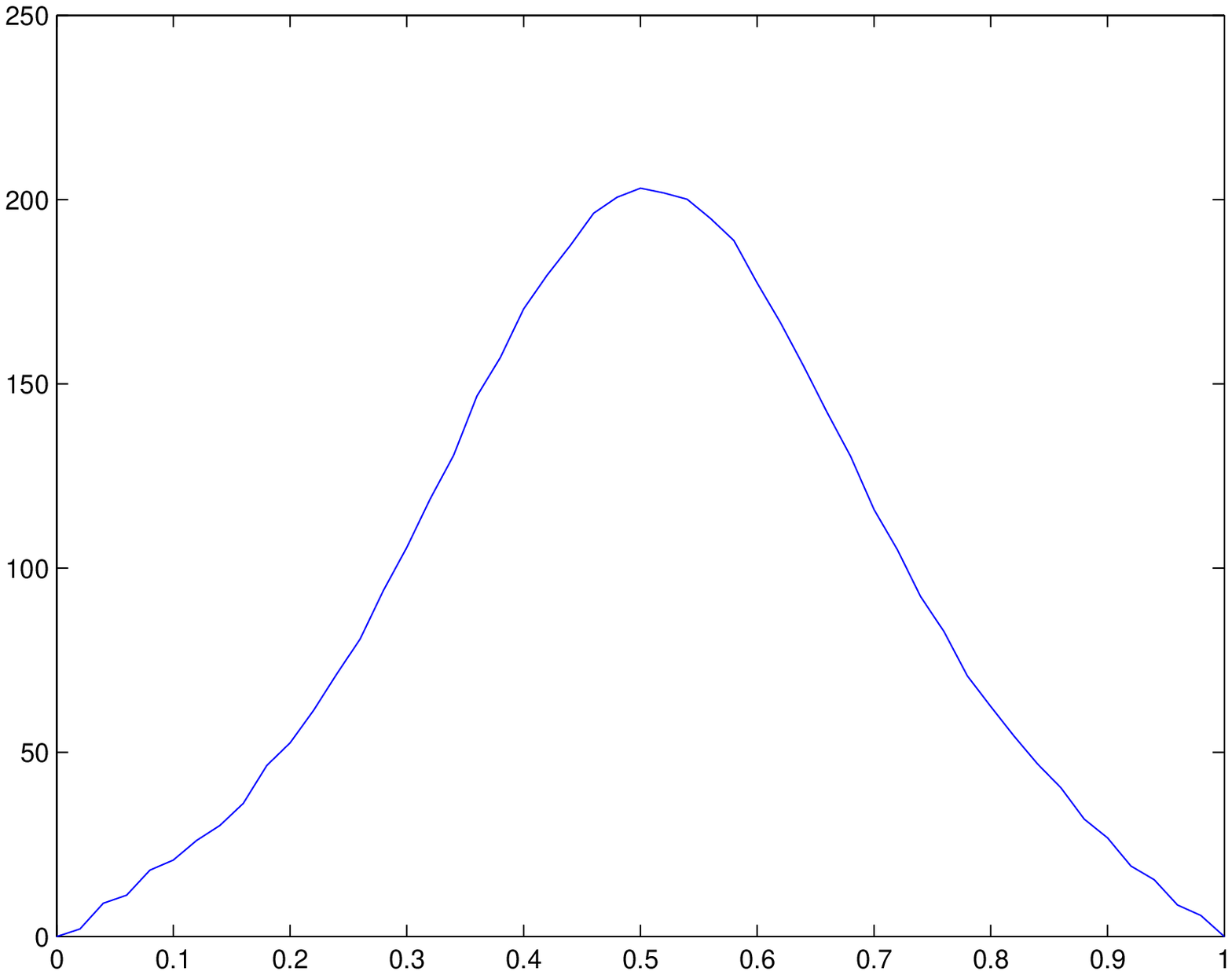} & \hspace{-6pt} \includegraphics[width=4cm]{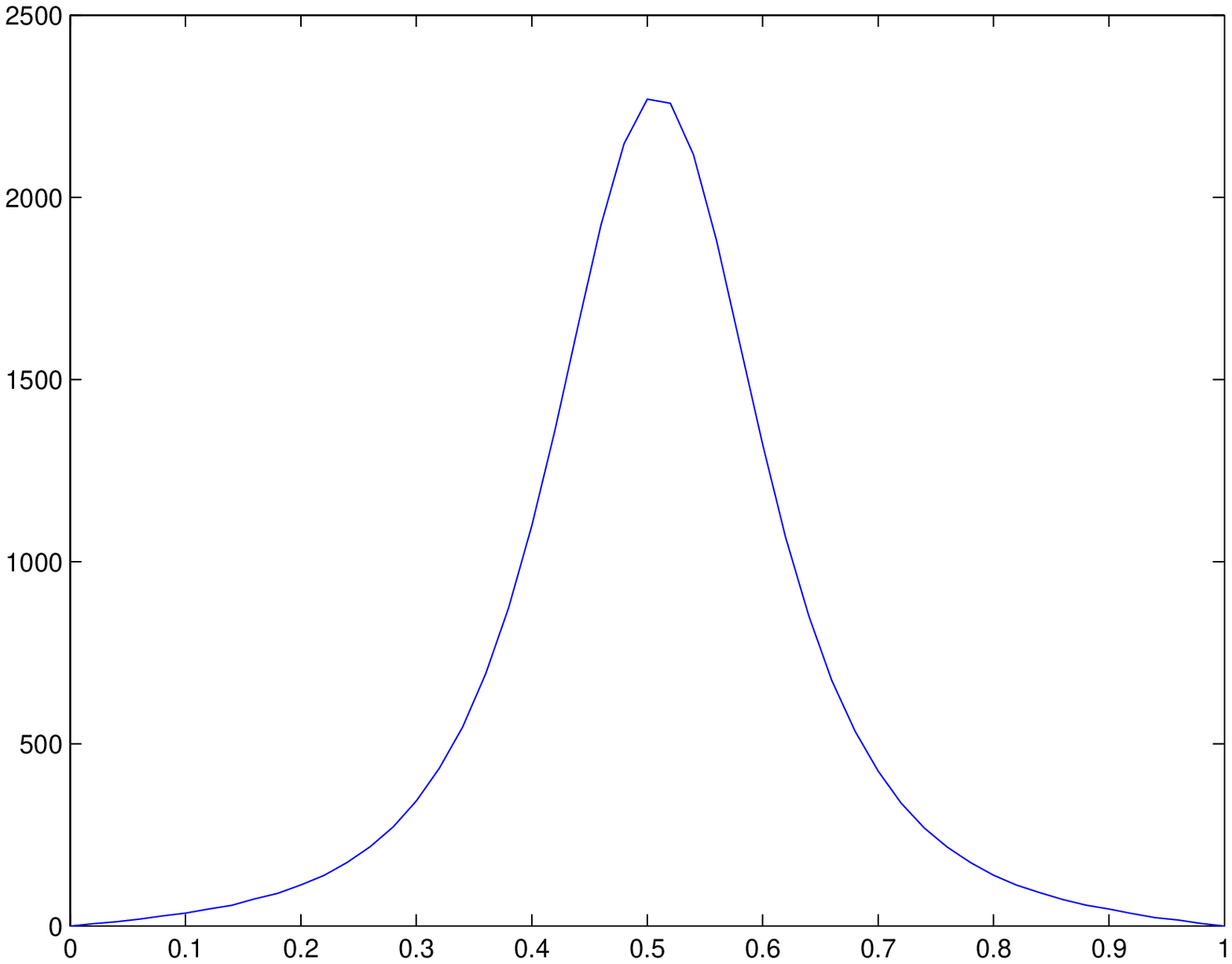} &
\hspace{-6pt}\includegraphics[width=4cm]{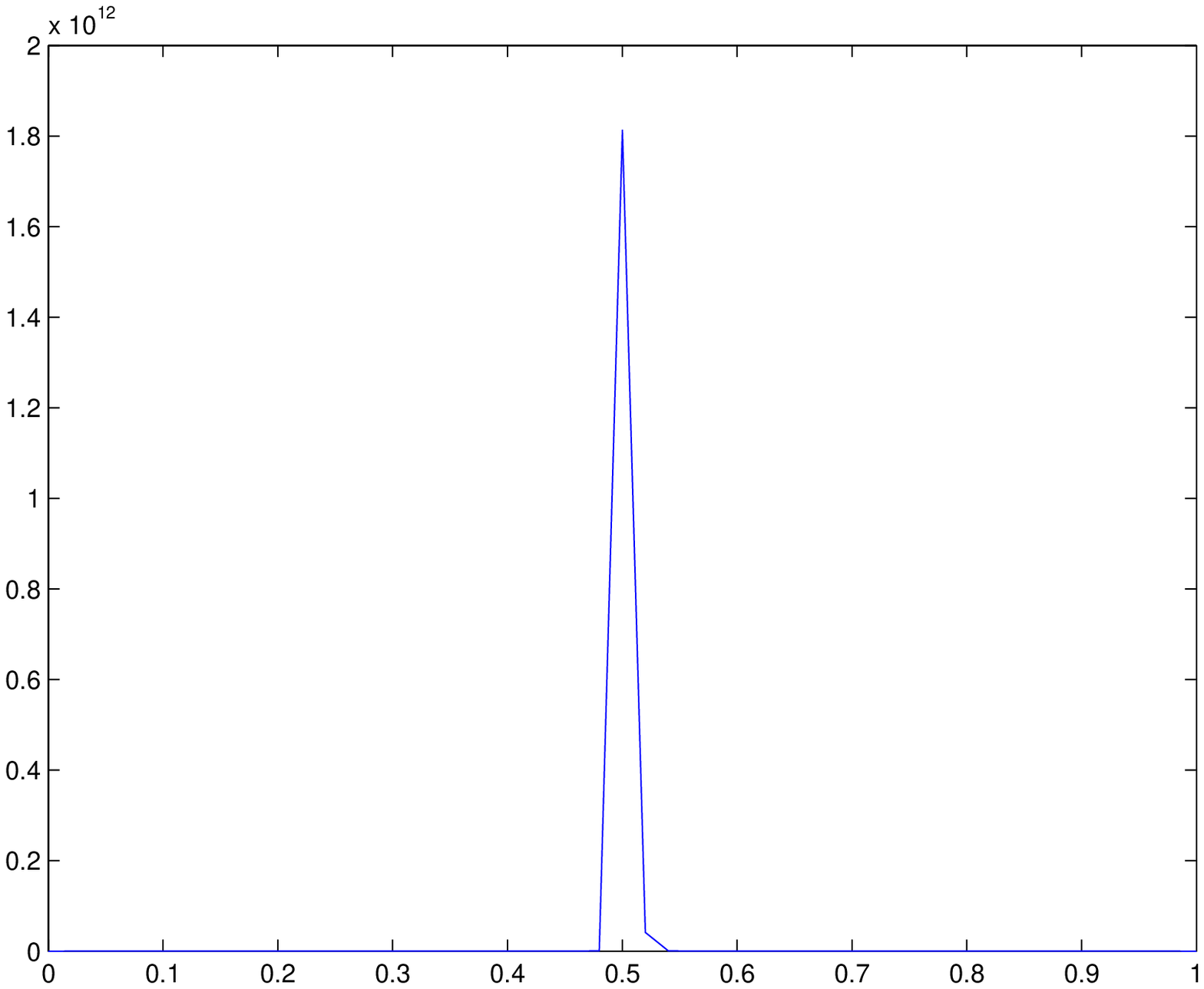} \\[-6pt]
\scriptstyle t=72.5012 & \scriptstyle t=72.5068 & \scriptstyle
t=72.5076
\end{array}
$$
\caption{Profiles of a sample solution at different times.} \label{perfiles}
\end{figure}
\begin{table}[ht]
\begin{tabular}{|c|r|r|}
\hline Snapshot & Time & $\|u(\cdot,t,\omega)\|_\infty$\\ \hline
 1 & 1.0000 & 5.6159\\\hline
 2 & 50.0000 &  3.3863\\\hline
 3 & 72.0202 & 15.5104\\\hline
 4 & 72.4202 & 18.2885\\\hline
 5 & 72.4802 & 38.5848\\\hline
 6 & 72.5002 & 82.8705\\\hline
 7 & 72.5012 & 203.0799\\\hline
 8 & 72.5068 & $2.2695 \times 10^3$\\\hline
 9 & 72.5076 & $1.8128 \times 10^{12}$\\\hline
\end{tabular}
\vspace{15pt} \caption{The maximum of the solution at differen
times} \label{tabla1}
\end{table}
\begin{figure}[ht]
\includegraphics[width=10cm]{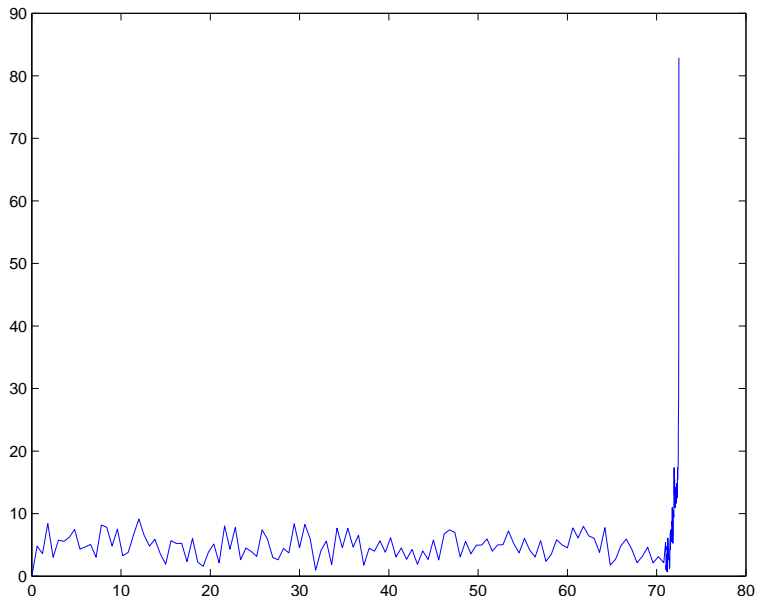}
\caption{The evolution of the maximum of a sample solution with
initial data $u_0\equiv 0$} \label{evo.maximo}
\end{figure}
\begin{figure}[ht]
\includegraphics[width=12cm]{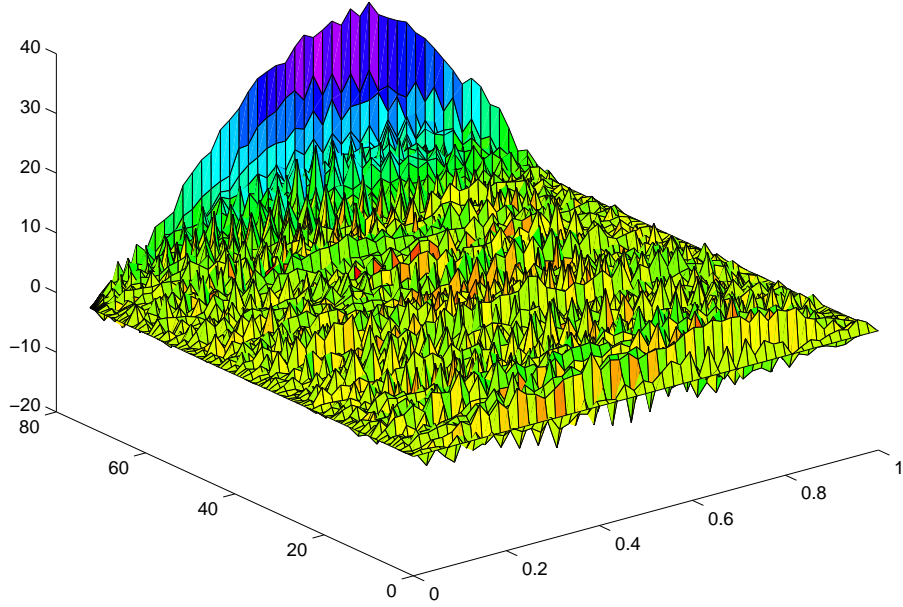}
\caption{The graph of a sample solution with initial datum
$u_0\equiv 0$}\label{3d}
\end{figure}
\begin{figure}[ht]
$$
\begin{array}{cc}
\includegraphics[width=6cm]{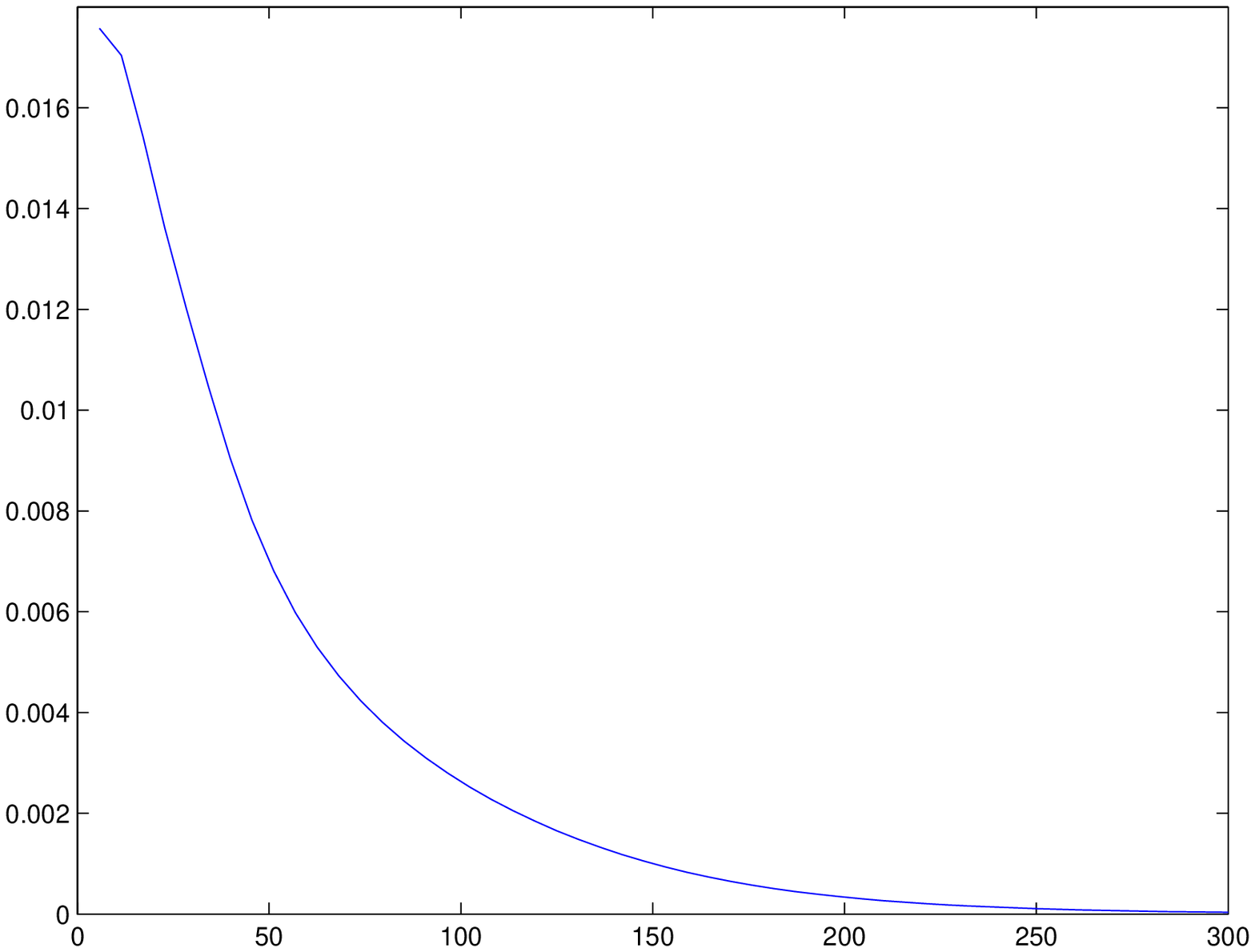}  & \hspace{-2pt}\includegraphics[width=6cm]{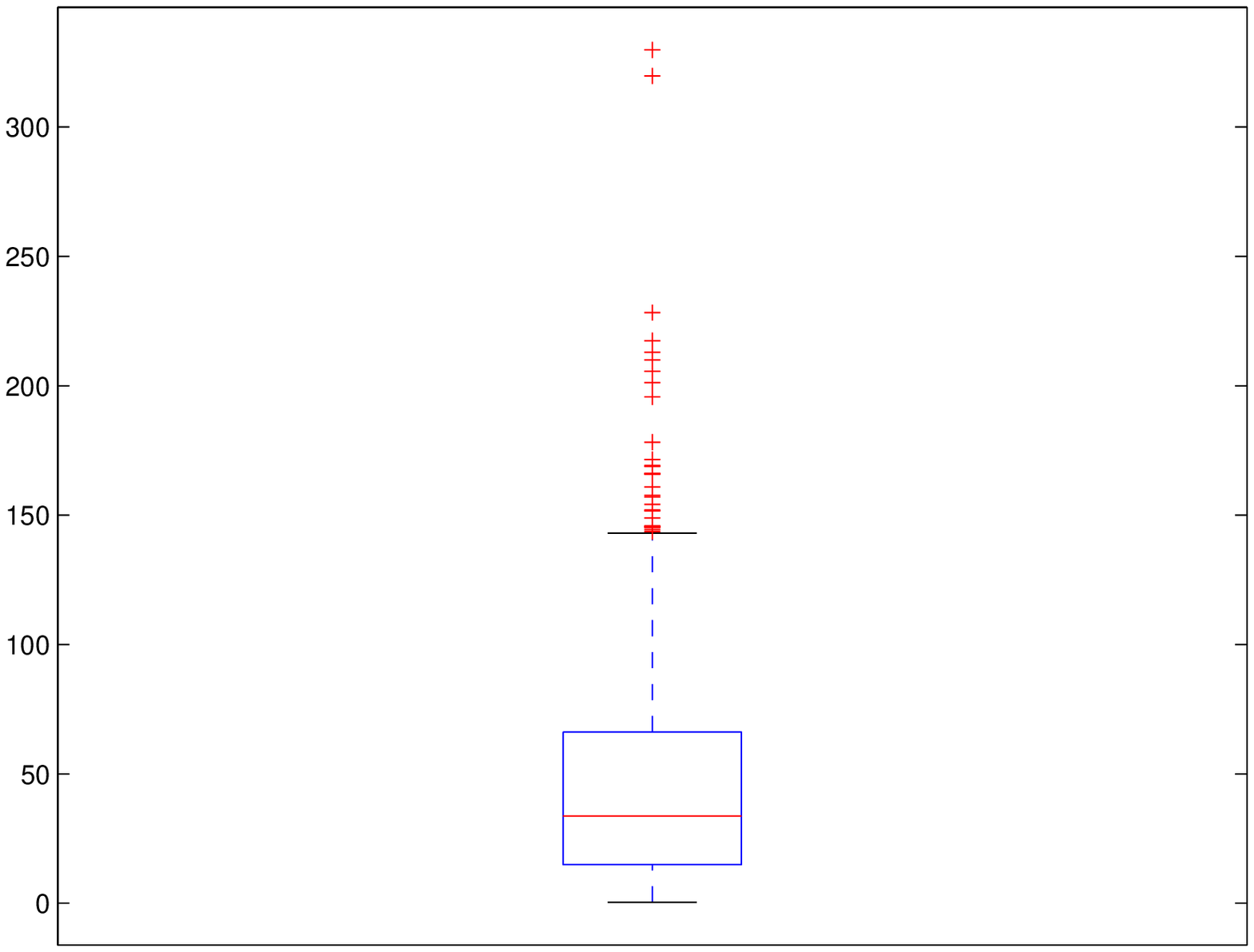}
\end{array}
$$
\caption{The kernel density estimator of the explosion time for
$\s=6.36$ and the corresponding box--plot.}\label{distrib}
\end{figure}

To perform the simulations we use the numerical scheme introduced in Section \ref{approximations}, that is we discretize the space variable with second order
finite differences in a uniform mesh of size $h=0.02$ (i.e.: $n=50$ nodes).
With this discretization we obtain a system of SDE that reads
\begin{equation*}
du_i = \frac{1}{h^2} (u_{i+1} -2u_i + u_{i-1}) dt + f(u_i) \, dt +
\frac{\s}{\sqrt{h}} \, dw_i, \quad 2 \le i \le n-1,
\end{equation*}
accompanied with the boundary conditions $u_1=u_n=0$, $u_i(0)=u_0(ih)$, $1 \le
i \le n$. The Brownian motions $w_i$ are obtained by space integration of the
Brownian sheet in the interval $[(i-1/2)h, (i+1/2)h)$.

To integrate this system we use an adaptive procedure similar to the one
developed in \cite{DFBGRS} for the one dimensional case. Here we adapt the time
step as in that work replacing the value of the solution (which is a real
number) by the $L^1-$norm of $u^j$, as is done in \cite{G} for the
deterministic case. More precisely, the totally discrete scheme reads as
follows
\begin{equation*}
u^{j+1}_i - u_i^j = \frac{\tau_j}{h^2} (u^j_{i+1} -2u^j_i + u^j_{i-1}) + \tau_jf(u^j_i) +
\frac{\s}{\sqrt{h}} (w_i(t^{j+1}) - w_i(t^j)),
\end{equation*}
accompanied with the boundary conditions $u_1^j=u_n^j=0$, for every $j\ge 1$
and $u_i^0=u_0(ih)$, $1 \le i \le n$. Here
$$
t_0 = 0, \quad \tau_j = \frac{\tau}{\sum_{i} h u_i^j}, \quad  t_{j+1} - t_j = \tau_j,
$$
and $\tau$ is the time-discretization parameter. The Brownian motions $w_i$ are
the ones of the semidiscrete scheme.

We want to remark that adaptivity in time is essential in this case since a
fixed time step procedure gives rise to globally defined approximations.

Concerning adaptivity in space, it is knwon for the case $\s=0$ that it is not
needed to capture the behavior of the maximal existence time. However spatial
adaptivity is needed to compute accurately the behavior of the solution near
the forming singularities (see \cite{BHR,FGR, FGR2, G}).

In spite that in Theorem \ref{teo.explota.continuo} we prove that solutions to
\eqref{SPDE} blow up with probability one for every $\s>0$ and every initial
data, we want to remark that it is not possible to observe that in numerical
simulations since for small $\s$, the explosion time is exponentially large
when the initial datum is small.

Essentially, in order to blow-up, the solution needs to be greater than the
positive stationary solution of the deterministic problem (i.e. the solution of
$v_{xx} = -f(v)$, which is of size 12 when $f(v)=(v_+)^2$) plus the order of
the noise $\sigma$. Once the solution is in that range of values, the noise
cannot prevent the explosion.

The probability $p_\sigma$ that such an event occurs in a finite fixed time
interval depends on $\s$ and is exponentially small ($p_\sigma \sim
\exp(-1/\s^2)$). Hence, one can estimate $P(T_\sigma>{\rm e}^{1/2\s^2}) \sim
\exp(\exp(-1/2\s^2))$. That means that for $\sigma$ small, explosions can not
be appreciated numerically and hence the importance of the theoretical
arguments.

So, to show the explosive behavior we choose to do the simulations with
$\s=6.36$ and initial datum $u_0\equiv 0$. We ran the code with $\s\le5$ until
time $t=1000$ and we did not observe explosions but a meta-stable behavior.

The features of a particular sample path are shown in Figure \ref{perfiles}.

Table \ref{tabla1} shows the times at where the solution is drawn and the
$L^\infty-$ norm of the solution at that time.

In Figure \ref{evo.maximo} we show the evolution of the $L^\infty$ norm and in
Figure \ref{3d} is the whole picture as a function of $x$ and $t$ of a sample
path.

Finally, Figure \ref{distrib} shows some statistics: we perform 832 simulations
of the solution with $\s=6.36$ to obtain a sample of the explosion time.
Actually, we stop the simulation when the maximum of the solution reaches the
value $10^{13}$. The kernel density estimator of the data obtained by the
simulation and the corresponding box--plot are shown. The sample mean is
46.8834 and the sample standard deviation 43.8857.

These statistics suggest that the distribution of the explosion time $T_\sigma$
is close to an exponential variable. This is confirmed by the metastable
nature of the phenomena. The expected behavior of $T_\sigma$ in this case is
$$
\lim_{\s  \to 0} \frac{T_\s}{\E(T_\s)} = Z,
$$
where $Z$ is a mean one exponential variable.


\begin{thebibliography}{10}
\bibitem{BB}
Catherine Bandle and Hermann Brunner.
\newblock Blowup in diffusion equations: a survey.
\newblock {\em J. Comput. Appl. Math.}, 97(1-2):3--22, 1998.

\bibitem{BP}
R.~Buckdahn and {\'E}.~Pardoux.
\newblock Monotonicity methods for white noise driven quasi-linear {SPDE}s.
\newblock In {\em Diffusion processes and related problems in analysis, Vol.\ I
  (Evanston, IL, 1989)}, volume~22 of {\em Progr. Probab.}, pages 219--233.
  Birkh\"auser Boston, Boston, MA, 1990.

\bibitem{BHR}  C. J. Budd, W. Huang and R. D. Russell.
\newblock Moving mesh methods for problems with blow-up.
\newblock {\em SIAM Jour. Sci. Comput.},  17(2):305--327,  1996.

\bibitem{CE}
Carmen Cort{\'a}zar and Manuel Elgueta.
\newblock Unstability of the steady solution of a nonlinear reaction-diffusion
  equation.
\newblock {\em Houston J. Math.}, 17(2):149--155, 1991.

\bibitem{DFBGRS}
Juan D{\'a}vila, Julian~Fern{\'a}ndez Bonder, Julio~D. Rossi, Pablo Groisman,
and Mariela Sued.
\newblock Numerical analysis of stochastic differential equations with
  explosions.
\newblock {\em Stoch. Anal. Appl.}, 23(4):809--825, 2005.

\bibitem{JLV}
Victor~A. Galaktionov and Juan~L. V{\'a}zquez.
\newblock The problem of blow-up in nonlinear parabolic equations.
\newblock {\em Discrete Contin. Dyn. Syst.}, 8(2):399--433, 2002.
\newblock Current developments in partial differential equations (Temuco,
  1999).

\bibitem{FGR} Ra{\'u}l Ferreira, Pablo Groisman and Julio D. Rossi.
\newblock Numerical blow-up for the porous medium equation with a source.
\newblock {\em  Numer. Methods Partial Differential Equations}, 20(4):552--575, 2004.

\bibitem{FGR2} Ra{\'u}l Ferreira, Pablo Groisman and Julio D. Rossi.
\newblock Adaptive numerical schemes for a parabolic problem with blow-up.
\newblock {\em IMA J. Numer. Anal.} 23(3):439--463, 2003.


\bibitem{G}
Pablo Groisman.
\newblock Totally discrete explicit and semi-implicit {E}uler methods for a
  blow-up problem in several space dimensions.
\newblock {\em Computing}, 76(3-4):325--352, 2006.

\bibitem{Gy}
Istv{\'a}n Gy{\"o}ngy.
\newblock Lattice approximations for stochastic quasi-linear parabolic partial differential equations driven by spae-time white noise I.
\newblock {\em Potential Analysis} 9:1--25, 1998

\bibitem{GP}
Istv{\'a}n Gy{\"o}ngy and {\'E}.~Pardoux.
\newblock On the regularization effect of space-time white noise on
  quasi-linear parabolic partial differential equations.
\newblock {\em Probab. Theory Related Fields}, 97(1-2):211--229, 1993.

\bibitem{KS}
Ioannis Karatzas and Steven~E. Shreve.
\newblock {\em Brownian motion and stochastic calculus}, volume 113 of {\em
  Graduate Texts in Mathematics}.
\newblock Springer-Verlag, New York, second edition, 1991.

\bibitem{MMR}
Xuerong Mao, Glenn Marion, and Eric Renshaw.
\newblock Environmental {B}rownian noise suppresses explosions in population
  dynamics.
\newblock {\em Stochastic Process. Appl.}, 97(1):95--110, 2002.

\bibitem{M1}
Carl Mueller.
\newblock Long-time existence for signed solutions of the heat equation with a
  noise term.
\newblock {\em Probab. Theory Related Fields}, 110(1):51--68, 1998.

\bibitem{M2}
Carl Mueller.
\newblock The critical parameter for the heat equation with a noise term to
  blow up in finite time.
\newblock {\em Ann. Probab.}, 28(4):1735--1746, 2000.

\bibitem{MS}
Carl Mueller and Richard Sowers.
\newblock Blowup for the heat equation with a noise term.
\newblock {\em Probab. Theory Related Fields}, 97(3):287--320, 1993.

\bibitem{Pardoux}
{\'E}.~Pardoux.
\newblock Spdes mini course given at fudan university, shanghai, april 2007.
\newblock 2007.
\newblock {\tt http://www.cmi.univ-mrs.fr/\~{}pardoux/spde-fudan.pdf}.

\bibitem{SGKM}
Alexander~A. Samarskii, Victor~A. Galaktionov, Sergei~P. Kurdyumov,
and
  Alexander~P. Mikhailov.
\newblock {\em Blow-up in quasilinear parabolic equations}, volume~19 of {\em
  de Gruyter Expositions in Mathematics}.
\newblock Walter de Gruyter \& Co., Berlin, 1995.
\newblock Translated from the 1987 Russian original by Michael Grinfeld and
  revised by the authors.

\bibitem{Walsh}
John~B. Walsh.
\newblock An introduction to stochastic partial differential equations.
\newblock In {\em \'Ecole d'\'et\'e de probabilit\'es de Saint-Flour,
  XIV---1984}, volume 1180 of {\em Lecture Notes in Math.}, pages 265--439.
  Springer, Berlin, 1986.

\end{thebibliography}
\end{document}